\documentclass[preprint,12pt]{elsarticle}
\usepackage{amssymb}
\usepackage{amsthm}
\usepackage{enumerate}
\usepackage{amsmath}
\usepackage{tikz}

\newcommand{\cl}[2][1]{{\operatorname{cl}}_{#1}\!\left(#2\right)} 
\newcommand{\inr}[2][1]{{\operatorname{int}}_{#1}\!\left(#2\right)}
\newcommand{\bd}[2][1]{{\operatorname{bd}}_{#1}\!\left(#2\right)}

\newtheorem{theorem}{Theorem}[section]

\newtheorem{coro}[theorem]{Corollary}
\newtheorem{prop}[theorem]{Proposition}

\newtheorem{note}[theorem]{Note}
\newtheorem{question}[theorem]{Question}
\newtheorem{obs}[theorem]{Observation}
\newtheorem{exa}[theorem]{Example}

\journal{a journal.}

\begin{document}

\begin{frontmatter}

\title{Minima Nonblockers and Blocked Sets of a Continuum}

\author[rvt]{C. Piceno}
\ead{cesarpicman@gmail.com}
\author[rvt1]{H. Villanueva}
\ead{hugo.villanueva@udlap.mx}

\address[rvt]{Facultad de Ciencias Físico Matemáticas, Benemérita Universidad Autónoma de Puebla, Avenida San Claudio y 18 Sur, Colonia San Manuel, Ciudad Universitaria, C.P. 72570, Puebla, México}

\address[rvt1]{Departamento de Actuar\'ia, F\'isica y Matem\'aticas, Universidad de las Am\'ericas Puebla. Sta. Catarina M\'artir. San Andr\'es Cholula, Puebla. C.P. 72810. M\'exico}

\begin{abstract}

Given a continuum $X$ and an element $x \in X$, $\pi(x)$ is the smallest set that contains $x$ and does not block singletons, and $B(x)$ is the set of all elements blocked by ${x}$. We prove that for each $x \in X$, $B(x)$ is connected, $B(x) \subset \pi(x)$, and that if $B(x)$ is closed, then $B(x)=\pi(x)$. Among other results, we prove that if $X$ is a Kelley continuum and $\pi(x)$ is proper, then $B(x)=\pi(x)$. Finally, we prove that for a certain class of dendroids, the family of minima non-blockers coincides with the family of connected non-blockers.
\end{abstract}


\end{frontmatter}

\section{Introduction}
\label{introduction}
Given a metric continuum $X$, an element $A \in 2^X - \{X\}$ is said to be a set that does not block
singletons of $X$ provided that for every point $x \in X - A$, there exists an order arc
$\alpha: [0,1] \to C(X)$ such that $\alpha(0)=\{x\}$, $\alpha(1)=X$ and $\alpha(t) \cap A = \emptyset$
for all $t \in [0,1)$. If $A$ does not block
singletons of $X$, we say that $A$ is a \textit{nonblocker} of $X$. The set that consists of all the nonblockers of $X$ is denoted by
$\mathcal {NB}(\mathcal F_1(X))$.
Given two elements $A, B \in 2^X$, we say that $B$ blocks $A$ provided that for each continuous path $\alpha : [0, 1]\to 2^X$ such that $\alpha(0) = A$ and $\alpha(1) = X$, there exists $t < 1$ such that $\alpha(t) \cap B 	\neq \emptyset$. 
 Given an element $a \in X$, the blocked set by $a$ is the set of all points of $X$ blocked by $\{a\}$ and it is denoted by $B(a)$.

In \cite{blockers} the concept of blockers is introduced, later, in \cite{camargo2}, \cite{camargo} and  \cite{nonblockers} general properties of the hyperspace of nonblockers were studied, in particular when it is a continuum.
In \cite{non-cut} the relationship between nonblockers and non-cut set hyperspaces is studied. 
Finally, in \cite{minimalnon} it is introduced the concept of minimun nonblocker set and in \cite{macias2}  the decomposition generated by the minima nonblocker sets is compared with other decompositions.


In Section \ref{aboutnonb},
we show some relations between minima nonblockers and 
the blocked sets. In particular, we show that in Kelley continua spaces, the blocked set and the minimum nonblocker set coincides when the blocked set is not dense. 
In Section \ref{Minima nonblockers in dendroids}, we characterize connected nonblockers for some classes of dendroids.

\section{Definitions and notation}
\label{definitions}
Given a subset $A$ of a metric space $Z$ with metric $d$, the closure, the boundary and the interior
of $A$ are denoted by $\cl[Z]{A}$, $\bd[Z]{A}$ and $\inr[Z]{A}$, respectively. Also,
$\mathcal V_r(A)$ denotes the open ball of radius $r$ about $A$.
A \textit{map} is a continuous function. A \textit{continuum} is a nonempty compact connected metric space.
A continuum is \textit{ homogeneous } if for any two points $p$ and $q$ of $X$, there exists
a homeomorphism $h$ from $X$ to itself such that $h(p)=q$.

Given a continuum $X$, we consider the following \textit{hyperspaces} of $X$:

$$2^X= \{A \subset X : A \text{ is nonempty and closed} \},$$
$$\mathcal C(X)= \{A \in 2^X : A \text{ is connected}\}$$
$$\mathcal F_1(X)= \{\{x\} : x \in X\},$$
These spaces are topologized with the Hausdorff metric defined as follows:
$$\mathcal H(A,B)= \text{ inf } \{\epsilon > 0 : A \subset \mathcal V_\epsilon (B)
\text{ and }B \subset \mathcal V_\epsilon (A)\}.$$

Clearly $\mathcal F_1(X) \subset \mathcal C(X) \subset 2^X$ and
$\mathcal F_1(X) \approx X$. It is known that if $X$
is a continuum then $2^X$, $\mathcal C(X)$ (Theorem 1.13 of \cite{nadlergringo})
are arcwise connected continua.



Let $X$ be a continuum. We say that $X$ has the \textit{Kelley's property}, provided that for each $\epsilon > 0$, there is a $\delta >0$ satisfying the following condition: if   $p,q \in X$ such that $d(p,q)< \delta$ and if $A \in \mathcal C(X)$ such that $p \in A$, then there exists a $B \in \mathcal C(X)$ such that $q \in B$ and $\mathcal{H}(A,B) < \epsilon$. It is known that every homogeneous continuum has the Kelley's property.

A \textit{path} in $2^X$ from $A$ to $B$ is a map $\gamma: [0,1] \to 2^X$ such that
$\gamma (0)=A$ and $\gamma(1)=B$.
An \textit{order arc} in $2^X$ is a map $\alpha: [0,1] \to 2^X$ such that if
$0 \leq s < t \leq 1$, then $\alpha(s) \subset \alpha(t)$ and $\alpha(s) \neq \alpha(t)$, we say that $\alpha$ is an order arc from $\alpha(0)$ to $\alpha(1)$.
By Lemma 1.11 of \cite{nadlergringo}, if $\alpha(0) \in \mathcal F_1(X)$, then $\alpha([0,1]) \subset \mathcal C(X)$.


Given a continuum $X$, an element $A \in 2^X - \{X\}$ is said to be a \textit{shore set} of $X$ if for each $\epsilon > 0$ there exists $B \in \mathcal C(X)$
 such that $\mathcal H(B,X)< \epsilon$ and $B \cap A = \emptyset$

According to Definition 0.1 of \cite{blockers}, we say that given two elements $A, B \in 2^X$,
$B$ \textit{does not block} $A$ if there exists a path $\gamma:[0,1] \to 2^X$ such that $\gamma(0)=A$, $\gamma(1)=X$
and $\gamma(t) \cap B = \emptyset$ for $t<1$. By Proposition 1.3 of \cite{blockers} and Theorem 1.8 of \cite{nadlergringo}, it is equivalent
to say that given two elements $A, B \in C(X)$, $B$ does not block $A$ if there exists an order arc
$\alpha:[0,1] \to C(X)$ such that $\alpha(0)=A$, $\alpha(1)=X$ and $\alpha(t) \cap B = \emptyset$ for $t<1$. 
We say that $B$ blocks $A$ if it is not true that $B$ does not block $A$.

We consider the following subspaces of $2^X$:
\begin{enumerate}[I.]
\item $\mathcal {NB}(\mathcal F_1(X))= \{A \in 2^X: A \text{ does not block the singletons of }X\}$;
\item $\mathcal S(X)= \{A \in 2^X: A \text{ is a shore set of }X\}$;
\item $\mathcal S_1(X)= \mathcal S(X) \cap \mathcal F_1(X)$; and 
\item $\mathcal {CNB(F}_1(X))=\mathcal {NB}(\mathcal F_1(X)) \cap \mathcal C(X).$

\end{enumerate}

It is clear that 
$\mathcal {NB}(\mathcal F_1(X)) \subset \mathcal S(X)$
and in Remark 3.3 of \cite{non-cut} the authors show that the inclusion can be proper.

Let $X$ be a continuum and $a$ an element of $X$, we define the family of nonblocker sets containing $a$ as $\mathcal{NB}(a,\mathcal{F}_1(X))=\{ A \in \mathcal {NB}(\mathcal F_1(X)): a \in A\}$ We define the minimum nonblocker function $\pi: X \to \mathcal{C} (X)$ as follows:
\[
\pi (a) =\left\{\begin{array}{ll}
\cap \mathcal{NB}(a,\mathcal{F}_1(X)), & \text{ if } \mathcal{NB}(a,\mathcal{F}_1(X))\neq\emptyset \\
X,     &  \text{ otherwise;}
\end{array}
\right.
\]
and let $$\Lambda_X=\{\pi (a): a \in X \}.$$

By Proposition 3.1 and Corollary 3.4 of \cite{minimalnon}, for each continuum $X$, $\Lambda_X \subset \mathcal {CNB(F}_1(X))$. Hence, for each $a\in X$, $\pi(a)$ is the minimum nonblocker set containing $a$. Thus, $\pi$ is a well-defined function. Recall that, for any continuum $X$ and $a\in X$, we define $$B(a)=\{ x : \{a\} \text{ blocks } x\}.$$
Notice that if $y \notin \pi(a)$, then $a$ does not block $y$. Hence, we have the following
\begin{obs}\label{bcpi}
Let $X$ be a continuum. Then $B(x) \subset \pi(x)$ for all $x \in X$.
\end{obs}


\section{Minima nonblockers and blocked sets}\label{aboutnonb}
In this section we study the blocked sets and their relation with minima nonblockers. We start with the following easy result.

\begin{prop}\label{Bcerrado}
Let $X$ be a continuum and $x\in X$ such that $B(x)$ is closed.
Then, $B(x)=\pi(x)$.
\end{prop}
\begin{proof}
By Observation \ref{bcpi}, it suffices to show that $\pi(x) \subset B(x)$. It is clear that $B(x)$ does not block any element outside of $B(x)$. Since $B(x)$ is closed, we have that $x \in B(x) \in \mathcal{NB(F}_1(X))$, and therefore $\pi(x) \subset B(x)$.
\end{proof}

It is not always true that $B(x)$ is closed, as Example \ref{mainexa} shows; however, $B(x)$ is always connected.

\begin{theorem}
For every continuum $X$ and every $x\in X$, $B(x)$ is connected.
\end{theorem}

\begin{proof}
Let $x\in X$ and suppose that $B(x)$ is not connected. Let $D$ be the component of $B(x)$ containing $x$ and let $C$ be a component of $B(x)$ such that $C\neq D$. If $\cl[X]{C} \cup D$ is connected, then there exists $y \in \cl[X]{C} - B(x)$. Therefore, there exists an order arc $\alpha_1: [0,1] \to C(X)$ such that $\alpha_1(0)=\{y\}$ and $x \notin \alpha_1(t)$ for all $t < 1$. Now, if $c \in C$, by Theorem 1.8 of \cite{nadlergringo}, there exists an order arc $\alpha_2: [0,1] \to C(X)$ such that $\alpha_2(0)=\{c\}$ and $\alpha_2(1)=\cl[X]{C}$. Hence,
\[
\alpha (t) =\left\{\begin{array}{ll}
\alpha_2(t), & \text{ if } t\in [0,1] \\
\alpha_1(t-1) \cup \cl[X]{C},     &  \text{ if } t \in [1,2]
\end{array}
\right.
\]
is a path from $c$ to $X$ avoiding $x$, a contradiction.
On the other hand, suppose that $\cl[X]{C} \cup D$ is not connected.
Then, there exists an open set $U$ such that $\cl [X]C \subset U$ and $U \cap D = \emptyset$.
Now, Theorem 1.8 of \cite{nadlergringo} implies there exists an order arc $\gamma: [0,1] \to C(X)$ such that $\gamma(0)=\cl[X]{C}$ and $\gamma(1)=X$. By continuity, we can take $t>0$ such that $\gamma(t)\subset U$. Thus, there exists $y \in \gamma(t)- \cl[X]{C} $ such that $y\in\gamma(t)- B(x)$. Since $x$ does not block $y$, we can construct a path as above and conclude that $x$ does not block $c$ for all $c \in C$. Again, a contradiction. 
\end{proof}

\begin{prop}\label{asubset}
Let $X$ be a continuum and $x \in X$. If $A \in \mathcal{C}(X)$ is such that $B(x) \cap A \neq \emptyset$, then $A \subset B(x)$ or $x \in A$.
\end{prop}
\begin{proof}
If $A \not\subset B(x)$ and $x \notin A$, take $y \in A - B(x)$. Then, for $z \in A \cap B(x)$, there exists an order arc from $z$ to $A$. Since $y \notin B(x)$, there exists an order arc from $y$ to $X$ avoiding $x$. With these two order arcs, we can construct a path from $z$ to $X$ avoiding $x$. A contradiction.
\end{proof}

The following immediate corollary shows that if $B(x)=\pi(x)$, for all $x \in X$ and $\Lambda_X$ is a decomposition of $X$, then each element of $\Lambda_X$ is a terminal subcontinuum of $X$.
\begin{coro}\label{terminal}
Let $X$ be a continuum such that for every $x \in X$, $B(x)=\pi(x)$ and $\Lambda_X$ is a decomposition of $X$.
If $A \in \mathcal C(X)$ and $A\cap\pi(x)\neq\emptyset$, then $A \subset \pi(x)$ or $\pi(x) \subset A$.
\end{coro}

Observe that, if $W$ is the Warsaw circle (\cite[Example 1.6]{nadler}), then $\pi(x)=B(x)$ for all $x \in W$, but not every $\pi(x) \in \Lambda_W$ is terminal. In this case, $\Lambda_W$ is not a decomposition.

The following corollaries say that every $\pi(x)$ and $\pi(y)$ are comparable whenever either $B(x)\cap\pi(y)\neq\emptyset$ or $B(x)\cap\pi(y)\neq\emptyset$ or $B(x)\cap B(y)\neq\emptyset$.

\begin{coro}\label{contenciondepis}
Let $X$ be a continuum and $x, y \in X$ such that $B(x) \cap \pi(y) \neq \emptyset$. Then, $\pi(x) \subset \pi(y)$ or $\pi(y) \subset \pi(x)$.
\end{coro}


\begin{coro}
Let $X$ be a continuum and $x, y \in X$ such that $B(x) \cap B(y) \neq \emptyset$. Then, $\pi(x) \subset \pi(y)$ or $\pi(y) \subset \pi(x)$.
\end{coro}

Example \ref{mainexa} shows that it is not always true that $\pi(x) \subset \pi(y)$ or $\pi(y) \subset \pi(x)$ whenever $\pi(x) \cap \pi(y) \neq \emptyset$.

It is natural to ask when $\pi(x) \cap \pi(y) \neq \emptyset$ implies that $\pi(x) \subset \pi(y)$ or $\pi(y) \subset \pi(x)$. Corollary \ref{contenciondepis} suggests the following question.
\begin{question}
For which continua $X$ it is true that $\pi(x)\cap\pi(y)\neq\emptyset$ implies $\pi(y)\cap B(x)\neq\emptyset$ or $\pi(x)\cap B(y)\neq\emptyset$?
\end{question}

Similar questions can be asked about blocked sets. Corollary \ref{comp2} gives a partial positive answer on comparability. First, we need the following result.

\begin{prop}\label{transitive}
Let $X$ be a continuum and $x, y, z \in X$ such that $x \in B(y)$ and $y \in B(z)$. Then, $x \in B(z)$.
\end{prop}
\begin{proof}
If $\alpha_1: [0,1] \to C(X)$ is an order arc, such that $\alpha(0)=\{x\}$ and $\alpha(1)=X$, then, there exists $t_0 \in(0,1)$ such that $y \in \alpha(t_0)$. Since $z$ blocks $y$, there exists $t \in (t_0, 1)$ such that $z \in \alpha(t)$. This implies that $x \in B(z)$.
\end{proof}

\begin{coro}\label{comp2}
Let $X$ be a continuum and $x, y \in X$ such that $B(x) \cap B(y) \neq \emptyset$. Then, $B(x) \subset B(y)$ or $B(y) \subset B(x)$ or $\pi(x)=\pi(y)$.
\end{coro}
\begin{proof}
By Proposition \ref{transitive}, if $x \in B(y)$ or $y \in B(x)$, then $B(x) \subset B(y)$ or $B(y) \subset B(x)$, respectively, and we are done. Suppose that $y \notin B(x)$, $x\notin B(y)$ and let $z \in B(x) \cap B(y)$. Let $\alpha$ be an order arc such that $\alpha(0)=\{z\}$, $\alpha(1)=X$. Then, there exists $t_0 < 1$ such that $x \in \alpha(t_0)$. If $y \notin \alpha(t_0)$, then $x$ must be in $B(y)$, otherwise, $z \notin B(y)$. Hence, for every $z \in B(x) \cap B(y)$ and every order arc $\alpha: [0,1] \to C(X)$ such that $\alpha(0)=\{z\}$ and $\alpha(1)= X$, if $t < 1$ is such that $x \in \alpha(t)$, then $y \in \alpha(t)$. This implies that $x \in \cl [X]{B(x)\cap B(y)}$ and $y \in \cl [X]{B(x)\cap B(y)}$. Therefore $\pi(x)=\pi(y)$.
\end{proof}

\begin{exa}\label{mainexa}
For each positive integer $n$, let $a_n=\left(0,\frac{1}{n}\right)$, $b_n=-a_n$, $A_n$ the line segment, in $\mathbb{R}^2$, joining $v=(-1,0)$ and $a_n$, and $B_n$ the line segment joining $w=(1,0)$ and $b_n$. Let $A_0=[-1,0]\times\{0\}$ and $B_0=[0,1]\times\{0\}$, $Z_1=\bigcup\limits_{n=0}^{\infty}A_n$, $Z_2=\bigcup\limits_{n=0}^{\infty}B_n$ and $Z=Z_1\cup Z_2$. Let $Y=R\cup Z$ be a compactification of a ray $R=h([0,\infty))$ with $Z$ as the remainder, where $h:[0,\infty)\to Y$ is an embedding. Let $\alpha$ and $\beta$ be arcs joining $h(0)$ with $p=\left(-\frac{1}{2},0\right)$ and $q=\left(\frac{1}{2},0\right)$ such that $\alpha\cap\beta=\{h(0)\}$. Let $X= Y \cup \alpha \cup \beta$ and $o=(0,0)$.
\end{exa}

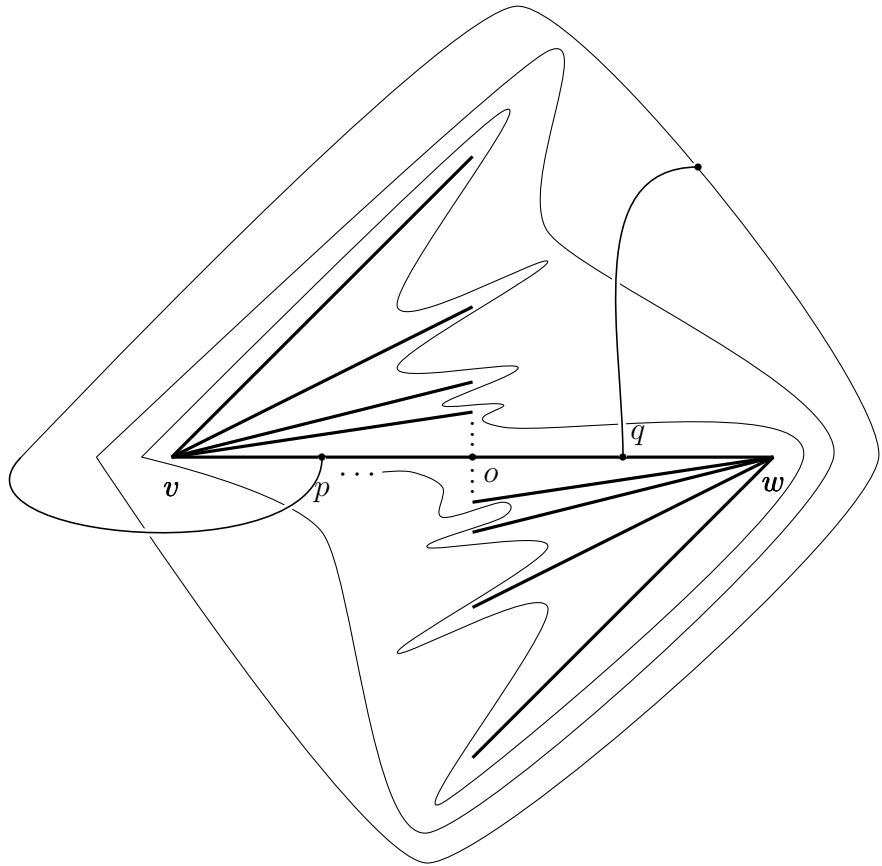
\[
\begin{tikzpicture}[scale=2]
\draw[very thick] (-2,0)--(2,0);
\foreach \y in {2,1,0.5, 0.3}{
    \draw[very thick] (-2,0) -- (0,\y) node at (-2,0)[below=5 pt]{$v$};
}
\foreach \y in {-2,-1,-0.5, -0.3}{
    \draw[very thick] (2,0) -- (0,\y) node at (2,0)[below=3 pt]{$w$};
}
\draw plot [smooth, tension=0.6] coordinates {(-2.2,0) (0.2,2.3) (-0.5,1) (0.5,1.3) (-0.5,0.6) (0.3,0.6) (-0.2,0.35) 
(0.2,0.35) (0.2,0.2)
(2.2,0) (-0.2,-2.3) (0.5,-1) (-0.5,-1.3) (0.5,-0.6) (-0.3,-0.6) (0.2,-0.4) (0.2,-0.3) (-0.2,-0.4) (-0.2,-0.2) 
(-0.4,-0.1) (-0.6,-0.1) };
\draw plot [smooth, tension=0.4] coordinates {(-2.2,0) (-1,-0.5) (-0.3,-2.5) (2.4,0) (0.5,1.5) (0.5,2.7) (-2.5,0)};
\draw plot[smooth, tension=0.3] coordinates {(-2.5,0) (-0.3,-2.7) (2.7,0) (0.3,3) (-3,0)};
\draw[draw=white, double=black, very thick] (-3,0) to[out=225,in=270] (-1,0) node{\tiny \textbullet} node at (-1,0)[below=5 pt]{$p$};
\draw[draw=white, double=black, very thick] (1,0) to[out=90,in=180] (1.5,1.93) node{\tiny$\bullet$} node at (1.1,0)[above]{$q$};
\draw node at (0,0){\tiny \textbullet}
        node at (0,0)[above]{$\vdots$} node at (0,0.1)[below]{$\vdots$}
        node at (-0.75,-0.12){$\cdots$}
        node at (1,0){\tiny\textbullet}node at (0,-0.12)[right]{$o$};
\end{tikzpicture}
\]

In Example \ref{mainexa}, $\pi(p) \cap \pi(q) \neq \emptyset$ but $\pi(p) \not\subset \pi(q)$ and $\pi(q) \not\subset \pi(p)$. In this case, $B(p) \cap B(q)= \emptyset$. It is interesting to notice that $\pi(x)=B(x)$ for all $x \in (p,q)$.

\begin{exa}\label{secondexa}
Let $Y$ be the continuum described in Example \ref{mainexa} and $X$ be the quotient space of $Y$ obtained by identifying the points $h(0)$ and $o$.
\end{exa}

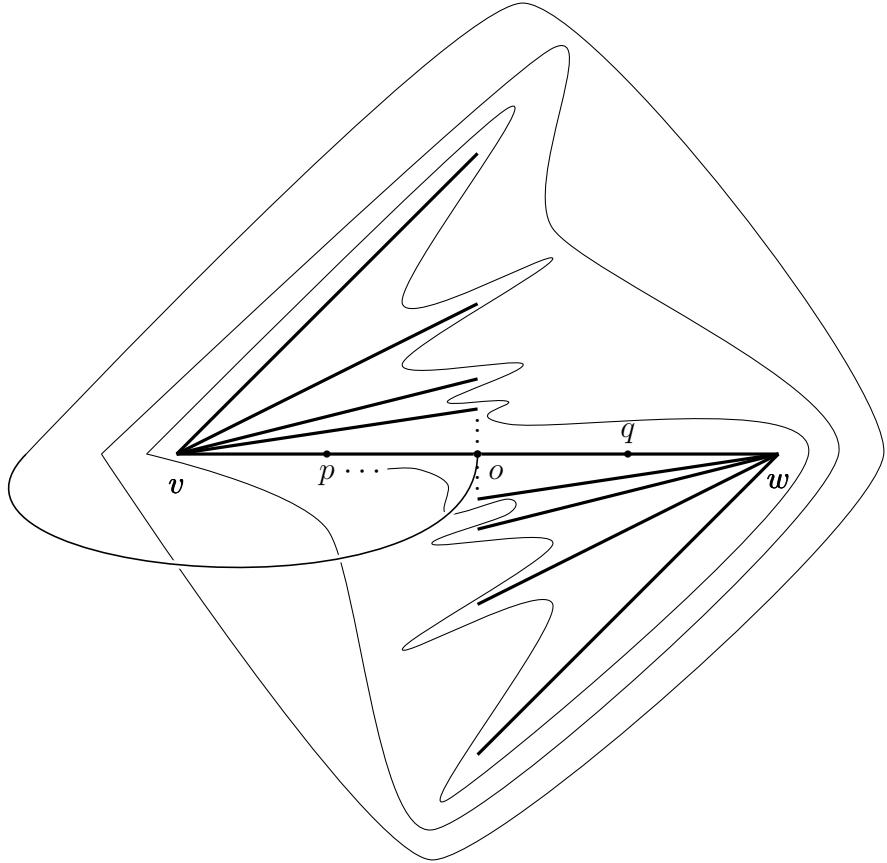
\[
\begin{tikzpicture}[scale=2]
\draw[very thick] (-2,0)--(2,0);
\foreach \y in {2,1,0.5, 0.3}{
    \draw[very thick] (-2,0) -- (0,\y) node at (-2,0)[below=5 pt]{$v$};
}
\foreach \y in {-2,-1,-0.5, -0.3}{
    \draw[very thick] (2,0) -- (0,\y) node at (2,0)[below=3 pt]{$w$};
}
\draw plot [smooth, tension=0.6] coordinates {(-2.2,0) (0.2,2.3) (-0.5,1) (0.5,1.3) (-0.5,0.6) (0.3,0.6) (-0.2,0.35) 
(0.2,0.35) (0.2,0.2)
(2.2,0) (-0.2,-2.3) (0.5,-1) (-0.5,-1.3) (0.5,-0.6) (-0.3,-0.6) (0.2,-0.4) (0.2,-0.3) (-0.2,-0.4) (-0.2,-0.2) 
(-0.4,-0.1) (-0.6,-0.1) };
\draw plot [smooth, tension=0.4] coordinates {(-2.2,0) (-1,-0.5) (-0.3,-2.5) (2.4,0) (0.5,1.5) (0.5,2.7) (-2.5,0)};
\draw plot[smooth, tension=0.3] coordinates {(-2.5,0) (-0.3,-2.7) (2.7,0) (0.3,3) (-3,0)};
\draw[draw=white, double=black, very thick] (-3,0) to[out=225,in=270] (0,0) node{\tiny \textbullet};
\draw node at (0,0){\tiny \textbullet}
        node at (0,0)[above]{$\vdots$} node at (0,0.1)[below]{$\vdots$}
        node at (-0.75,-0.12){$\cdots$}
        node at (0,-0.12)[right]{$o$}
        (-1,0) node{\tiny \textbullet}
        node at (-1,0)[below]{$p$}
        (1,0) node{\tiny \textbullet}
        node at (1,0)[above]{$q$};
        
\end{tikzpicture}
\]

In Example \ref{secondexa}, $\pi(z) = \pi(o)=B(o)$ for every $z \in Z$, but $B(z) \subsetneq \pi(z)$ for each $z\in Z - \{o\}$. Moreover, it is not difficult to see that $cl(B(z))\subsetneq \pi(z)$ for each $z\in vw-\{o\}$; thus, $B(z)$ is not dense in $\pi(z)$. Also, it is clear that $B(p) \cap B(q)= \emptyset$. 


It is natural to ask for which continua $X$ and which point $x\in X$, it is true that $\pi(x)=B(x)$. We prove that for Kelly continua $X$,  
the minimum nonblocker and the blocked set are equal for those points $x\in X$ such that $\text{cl}B(x)$ is not dense. 





\begin{theorem}\label{bequalpi}
Let $X$ be a Kelley continuum. 
Then $\pi (x)= B(x)$ for every $x \in X$ such that $\operatorname{cl}(B(x))\neq X$.
\end{theorem}
\begin{proof}
If $B(x)$ is closed, then $\pi(x)=B(x)$ by Proposition \ref{Bcerrado}.
Suppose $\pi(x)-B(x)\neq\emptyset$  and choose $z\in (\operatorname{bd}(B(x))\cap \pi(x))-B(x)$.
If $\operatorname{cl}(B(x)) \neq X$, there exists $Z \in \mathcal C(X)$ such that $z \in Z$, $Z - \operatorname{cl}(B(x)) \neq \emptyset$ and $x \notin Z$. Let $\epsilon = min\{\{d(w, x): w \in Z\} \cup \{ d_H(Z, \operatorname{cl}(B(x)))\}\}$. Since $X$ is a Kelley continuum, there exists $y \in B(x)$ and $Y \in \mathcal C(X)$ such that $y \in Y$ and $d_H(Z,Y) < \frac{\epsilon}{3}$. Hence, $x \notin Y$ and $Y \not \subset B(X)$ which contradicts Proposition \ref{asubset}.
\end{proof}

If it could be proven that in Kelley continua there is no $x$ such that $B(x)$ is dense and different from $X$, equality would hold for all points in Theorem \ref{bequalpi}. Therefore, we pose the following question: 

\begin{question}
If $X$ is a Kelley continuum, is it true that $\operatorname{cl}(B(x)) \neq X$ or $B(x) = X$ for every $x \in X$?
\end{question}

\begin{coro}\label{bequalpihomo}
If $X$ is a homogeneous continuum such that $\mathcal {NB}(\mathcal F_1(X)) \neq
\emptyset$, then $\pi (x)= B(x)$ for every $x \in X$.
\end{coro}
\begin{proof}
Let $a \in A \in \mathcal {NB}(\mathcal{F}_1(X))$ and $x \in X$. Let $h: X \to X$ be a homeomorphism such that $h(a)=x$. By Proposition 3.9 from \cite{minimalnon}, $h(A) \in \mathcal {NB}(\mathcal{F}_1(X))$. Hence, $\operatorname{cl}(B(x)) \subset \pi(x) \subset h(A) \neq X$. Since, $X$ is a homogeneous continuum, X is a Kelley continuum and therefore $\pi(x)= B(x)$.

\end{proof}

We notice that in the proof of Theorem 4.3 of \cite{minimalnon}, it is assumed that $\pi(x)=B(x)$ for each $x\in X$, and we have seen that this is not always the case. By Theorem \ref{bequalpi}, it can be given a complete proof of the mentioned theorem.




Since homogeneous continua are Kelley, the next corollary  follows from \cite[Theorem 4.2]{minimalnon} and Corollary 3.4.

\begin{coro}\label{componentesmin} 
If  $X$ is a homogeneous continuum such that $\mathcal {NB}(\mathcal{F}_1(X))\neq\emptyset$, then $\Lambda_X$ is a continuous decomposition into maximal terminal subcontinua $\pi(x)=B(x)$. The decomposition space $Y$   is a homogeneous continuum with    $\Lambda_Y\approx\mathcal{F}_1(Y)$.
\end{coro}

The following is an example of a continuum $X$ such that $\Lambda_X$ is a decomposition, $\pi(x) = B(x)$ for every $x \in X$, but it is neither upper nor lower semicontinuous.

\begin{exa}
For each positive integer $n$, let $L_n=\{\frac{1}{n}\}\times[-1,1]$ and let $L_0=\{0\}\times[-1,1]$. Let $Y_n$ be a compactification of the real line, embedded in $[\frac{1}{n+1},\frac{1}{n}]\times[-1,1]$, such that $L_n$ and $L_{n+1}$ are the two components of the remainder of $Y_n$ and let 
$$Y=L_0\cup\left(\bigcup\limits_{n=1}^{\infty}Y_n\right).$$

Let $A$ and $B$ be arcs in the plane joining $(0,1)$ to $(1,1)$ and $(0,-1)$ to $(1,-1)$, respectively, such that $A\cap Y=\{(0,1),(1,1)\}$. Let $X=A\cup B\cup Y$. 

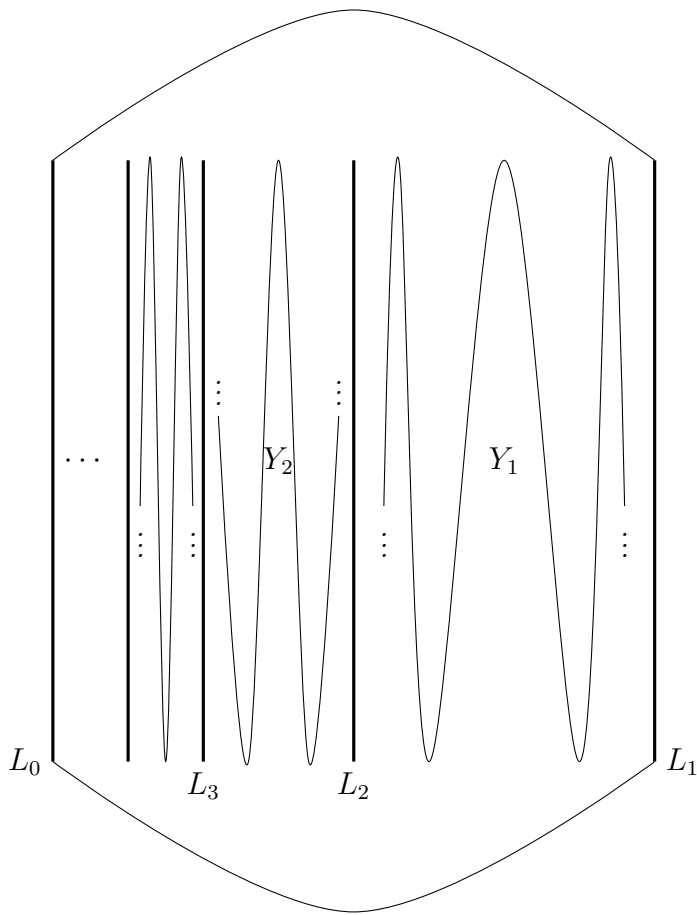
\[
\begin{tikzpicture}[scale=2]
\draw[very thick] (-2,-2)--(-2,2);
\foreach \x in {2,0,-1, -1.5}{
    \draw[very thick] (\x,-2) -- (\x,2);
}
\draw plot [smooth, tension=0.6] coordinates {(1.8,-0.3) (1.7,2) (1.5,-2) (1,2) (0.5,-2) (0.3,2) (0.2,-0.3) };
\draw plot [smooth, tension=0.6] coordinates {(-0.1,0.3) (-0.3,-2) (-0.5,2) (-0.7,-2) (-0.9,0.3) };
\draw plot [smooth, tension=0.6] coordinates {(-1.07,-0.3) (-1.15,2) (-1.25,-2) (-1.35,2) (-1.42,-0.3) };
\draw plot [smooth, tension=0.6] coordinates {(-2,2) (0,3) (2,2)};
\draw plot [smooth, tension=0.6] coordinates {(-2,-2) (0,-3) (2,-2)};
\draw node at (1.8,-0.3)[below]{$\vdots$};
\draw node at (0.2,-0.3)[below]{$\vdots$};
\draw node at (-0.9,0.3)[above]{$\vdots$};
\draw node at (-0.1,0.3)[above]{$\vdots$};
\draw node at (-1.42,-0.3)[below]{$\vdots$};
\draw node at (-1.07,-0.3)[below]{$\vdots$};
\draw node at (-2,0)[right]{$\cdots$};
\draw node at (1,0){$Y_1$};
\draw node at (-0.5,0){$Y_2$};
\draw node at (-2,-2)[left]{$L_0$};
\draw node at (2,-2)[right]{$L_1$};
\draw node at (0,-2)[below]{$L_2$};
\draw node at (-1,-2)[below]{$L_3$};
\end{tikzpicture}
\]

Note that
$$\Lambda_X=\{L_n:n\in\mathbb{N}\}\cup\{\{x\}:x\notin L_n\text{ for each }n\in\mathbb{N}\}.$$

That is, $\Lambda_X$ is a decomposition which is neither upper nor lower semi-continuous.
\end{exa}

\begin{question}
For which continua $X$, $\Lambda_X$ is a (upper-semi)continuous decomposition?
\end{question}

\section{Minima nonblockers in dendroids} \label{Minima nonblockers in dendroids}
As we know, all minima nonblockers of a continuum $X$ are connected. An interesting question is when the set of connected nonblockers coincides with the set of minima nonblockers. Corollary \ref{componentesmin} give us sufficient conditions to ensure that $\Lambda_X=\mathcal{NB(F}_1(X))$. It is clear that if either $X$ is a continuum such that $\mathcal S(X)$ is finite or $X$ is a finite graph, then $\Lambda_X = \mathcal {CNB(F}_1(X))$. In this section, we study the case when $X$ is a dendroid such that $X - \mathcal S_1 (X) \neq \emptyset$.

We say that a continuum $X$ is \textit{unicoherent} if $A \cap B$ is connected for each subcontinua $A$ and $B$ such that $A  \cup B = X$; and $X$ is said to be \textit{hereditarily unicoherent} if every subcontinuum of $X$ is unicoherent. 
A \textit{dendroid} is an arcwise connected hereditarily unicoherent continuum.
A point $p$ in a dendroid $D$ is called a \textit{center} for $D$ if there exist two points $b$ and $c$ in $D$, called \textit{basin points}, such that for every $\epsilon > 0$ there exists a continuum $C$ containing $p$ of diameter less than $\epsilon$, called a \textit{bottleneck}, and there are open sets $U$ and $V$, containing $b$ and $c$, respectively, called \textit{basins}, such that every arc from $U$ to $V$ intersects $C$. In this case, we also say that $p$ is a $bc$-\textit{center}.
A point $p$ in a dendroid $D$ is a \textit{strong center} if there exist open sets $U$ and $V$ such that every arc in $D$ from $U$ to $V$ contains $p$.

\begin{note}
If $X$ is a dendroid and $A\in 2^X$, then $A \in \mathcal {NB}(\mathcal F_1(X))$ if and only if $A$ has empty interior and each arc component of $X - A$ is dense in $X$.
\end{note}

\begin{theorem}\label{arcoconexo}
Let $X$ be a dendroid such that $X - \mathcal S_1(X) \neq \emptyset$ and $A\in\mathcal{C}(X)$, then $A \in \mathcal {CNB(F}_1(X))$ if and only if $A$ has empty interior and $X - A$ is arc-connected.
\end{theorem}
\begin{proof}
It is clear that if $A \in \mathcal C(X)$ has empty interior and $X - A$ is arc-connected, then $A \in \mathcal {CNB}(\mathcal F_1(X))$. 

Now, let $A \in \mathcal{NB(F}_1(X))$.
By Theorem 3.6 of \cite{minc}, Theorem 1 of \cite{nall} and Theorem 3.2 of \cite{non-cut}, we have that there exists one center $o \in X - A$. Let $b$ and $c$ be two points of $X$ such that $o$ is a $bc$-center. Let $C \in C(X)$ a bottleneck such that $ o \in C$ and $C \cap A = \emptyset$. Let $U$ and $V$ be basins for $b$ and $c$ respectively. Since $\inr[X]{A}= \emptyset$, we have that $U - A \neq \emptyset$ and $V - A \neq \emptyset$. Hence, if $x \in X - A$ and $\alpha$ is an order arc $\alpha:[0,1] \to C(X)$ from $\{x\}$ to $X$ avoiding $A$, there exists $t \in [0, 1)$ such that $\alpha(t) \cap (U - A) \neq \emptyset$ and  $\alpha(t) \cap (V - A) \neq \emptyset$. Since $\alpha(t) \in C(X)$, $\alpha(t)$ is arc-connected and therefore $\alpha(t) \cap C\neq \emptyset$. Thus, $X - A$ is arc-connected.
\end{proof}


Since no strong center is a shore point, by Theorem \ref{arcoconexo} we have the next corollary.

\begin{coro}
Given a dendroid $X$ with a strong center and $A\in\mathcal{C}(X)$, then $A \in \mathcal {CNB}(\mathcal F_1(X))$ if and only if $A$ has with empty interior and $X - A$ is arc-connected.
\end{coro}

Moreover, by \cite[Theorem 3.11]{minc}, every planar dendroid contains a single point bottleneck and no point  bottleneck is a shore point. Hence, Theorem \ref{arcoconexo} implies the following.

\begin{coro}
If $X$ is a planar dendroid and $A\in \mathcal{C}(X)$, then, $A \in \mathcal {CNB(F}_1(X))$ if and only if $A$ has empty interior and $X - A$ is arc-connected.
\end{coro}


\begin{theorem}
Let $X$ be a dendroid such that $X - \mathcal S_1(X) \neq \emptyset$.
Then, $\Lambda_X = \mathcal {CNB(F}_1(X))$.
\end{theorem}
\begin{proof}
It suffices to show that $\mathcal {CNB(F}_1(X)) \subset \Lambda_X$. Let $A \in \mathcal {CNB(F}_1(X))$. Let $p \in A$, $q \in X - A$ and $R$ the arc from $p$ to $q$. By compactness, we can take $a\in R$ such that $aq\cap A=\{a\}$. If $\pi(a) \neq A$, we can take $s \in A - \pi(a)$. Since $\pi(a)$ does not block $s$ there exists a continuum $B$ such that $s \in B$ and $B \cap (X- A) \neq \emptyset$, $B \cap \pi(A)= \emptyset$. By Theorem \ref{arcoconexo}, there exists an arc from $q$ to any point of $B - A$ avoiding $A$. Hence, there is an arc from $s$ to $q$ avoiding $\pi(a)$. On the other hand, we can construct an arc from $s$ to $q$ passing trough $a$. This implies that $X$ contains a closed curve, which is a contradiction. Therefore $\pi(a) =A$.
\end{proof}

In \cite[Example 1.8]{puga} is given an example of a dendroid such that every point is a shore point. Then, we have the following question.

\begin{question}
If $X$ is a dendroid such that $X-\mathcal{S}_1(X)=\emptyset$, is it true that $\Lambda_X = \mathcal {CNB(F}_1(X))$?
\end{question}


\end{document}